\documentclass[12pt,a4paper,reqno,oneside]{amsart}
\usepackage{t1enc}
\usepackage{times}
\usepackage{amssymb}
\usepackage[mathscr]{euscript}
\usepackage{graphicx}
\usepackage{hyperref}

\usepackage{float}
\usepackage{epstopdf}

\newtheorem{thm}{Theorem}[section]
\newtheorem{prop}[thm]{Proposition}
\newtheorem{lem}[thm]{Lemma}
\newtheorem{cor}[thm]{Corollary}

\theoremstyle{remark}
\newtheorem{rem}[thm]{Remark}

\theoremstyle{definition}

\newcommand*{\rom}[1]{\expandafter\@slowromancap\romannumeral #1@} 
\renewcommand{\phi}{\varphi} 
\newcommand{\tr}{\mathrm{Tr}} 
\newcommand{\E}{\mathrm{E}} 
\newcommand{\A}{\mathfrak A} 
\newcommand{\F}{\mathcal F} 
\newcommand{\sgn}{\mathrm{sgn}} 
\renewcommand{\Re}{\mathrm{Re}} 
\renewcommand{\Im}{\mathrm{Im}} 

\renewcommand{\>}{\rangle}
\newcommand{\filt}{(\F_t)_{t\geq 0}} 
\newcommand{\fraum}{(\Omega,\filt,\A,P)} 
\newcommand{\skor}{\mathbb D(\mathbb R^d)} 
\newcommand{\canspace}{(\skor,\mathfrak D,(\mathcal D_t)_{t\geq 0})} 
\newcommand{\nooutput}[1]{}

\begin{document}

\title[Optimal bounds for densities of SDEs]{Optimal bounds for the densities of solutions of SDEs with measurable and path dependent drift coefficients}
\date{\today}

\author[Banos]{David Ba\~nos}
\address[David Ba\~nos]{\\
Department of Mathematics \\
University of Oslo\\
P.O. Box 1053, Blindern\\
N--0316 Oslo, Norway}
\email[]{davidru@math.uio.no}
\author[Kr\"uhner]{Paul Kr\"uhner}
\address[Paul Kr\"uhner]{\\
Department of Mathematics \\
University of Oslo\\
P.O. Box 1053, Blindern\\
N--0316 Oslo, Norway}
\email[]{paulkru@math.uio.no}


\thanks{This paper has been developed under financial support of the project "Managing Weather Risk in
Electricity Markets" (MAWREM), funded by the RENERGI-program of the Norwegian Research Council.}

\begin{abstract}
We consider a process given as the solution of a stochastic differential equation with irregular, path dependent and time-inhomogeneous drift coefficient and additive noise. Explicit and optimal bounds for the Lebesgue density of that process at any given time are derived. The bounds and their optimality is shown by identifying the worst case stochastic differential equation. Then we generalise our findings to a larger class of diffusion coefficients. 
\end{abstract}

\maketitle

\section{Introduction}

The study of regularity of solutions 
 of stochastic differential equations (SDEs) has been a topic of great 
 interest within stochastic analysis, especially since Malliavin calculus was founded. One of the main motivations of Malliavin calculus is precisely to study the regularity properties of the law of Wiener functionals, for instance, solutions to SDEs, as well as, properties of their densities. A classical result on this subject is that if the coefficients of an SDE are 
 $C^{\infty}$ functions with bounded derivatives of any order and the so-called H\"{o}rmander's condition (see e.g. \cite{hormander.69}) holds, then the solution of the equation is smooth in the Malliavin sense. Then P. Malliavin shows in \cite{malliavin.78} that the laws of the solutions at any time are absolutely continuous with respect to the Lebesgue measure and the densities are smooth and bounded. 
Another approach 
is attributed to N. Bouleau and F. Hirsch where they show in \cite{bouleau.hirsch.86} absolute continuity of the finite dimensional laws of solutions to SDEs based on a stochastic calculus of variations in finite dimensions where they use 
 a limit argument. Also, as a motivation of \cite{bouleau.hirsch.86}, D. Nualart and M. Zakai \cite{nualart.zakai.89} found related results on the existence and smoothness of conditional densities of Malliavin differentiable random variables.

It appears to be quite difficult to derive regularity properties for the densities of solutions to SDEs with singular coefficients, i.e. non-Lipschitz coefficients, in particular in the drift. Nevertheless, some findings in this direction have been attained. Let us for instance remark here the work by M. Hayashi, A. Kohatsu-Higa and G. Y\^{u}ki in \cite{hayashi.et.al.12} where the authors show that SDEs with H\"{o}lder continuous drift and smooth elliptic diffusion coefficients admit H\"{o}lder continuous densities at any time. Their techniques are mainly based on an integration by parts formula (IPF) in the Malliavin setting and estimates on the characteristic function of the solution in connection with Fourier's inversion theorem. Another result in this direction is due to S. De Marco in \cite{de.marco.11} where the author proves smoothness of the density on an open domain under the usual condition of ellipticity and that the coefficients are smooth on such domain. A remarkable fact is that H\"{o}rmander's condition is skipped in this proof. Moreover, estimates for the tails are also given. The technique relies strongly on Malliavin calculus and an IPF together with estimates on the Fourier transform of the solution. One may already observe that integration by parts formulas in the Malliavin context are a powerful tool for the investigation of densities of random variables as it is the case in the work by V. Bally and L. Caramellino in \cite{bally.caramellino.11} where an IPF is derived and the integrability of the weight obtained in the formula gives the desired regularity of the density. As a consequence of the aforesaid result \mbox{D. Ba\~{n}os} and T. Nilssen give in \cite{banos.nilssen.14} a criterion to obtain regularity of densities of solutions to SDEs according to how regular the drift is. The technique is also based on Malliavin calculus and a sharp estimate on the moments of the derivative of the flow associated to the solution. This result is a slight improvement of a very similar criterion obtained by S. Kusuoka and D. Stroock in \cite{kusuoka.stroock.82} when the diffusion coefficient is constant and the drift may be unbounded. Another related result on upper and lower bounds for densities is due to V. Bally and A. Kohatsu-Higa in \cite{bally.kohatsu-higa.10} where bounds for the density of a type of a two-dimensional degenerated SDE are obtained. For this case, it is assumed that the coefficients are five times differentiable with bounded derivatives. Finally, we also mention the results by A. Kohatsu-Higa and A. Makhlouf in \cite{kohatsu.makhlouf.13} where the authors show smoothness of the density for smooth coefficients that may also depend on an external process whose drift coefficient is irregular. They also give upper and lower estimates for the density.

It is worth alluding the exceptional result by A. Debussche and N. Fournier in \cite{debussche.fournier.13} on this topic where the authors show that the finite dimensional densities of a solution of an SDE with jumps lies in a certain (low regular) Besov space when the drift is H\"{o}lder continuous. The novelty is that their method does not use Malliavin calculus as in the aforementioned works.

It is therefore important to highlight that in this paper we do \emph{not} use Malliavin calculus or any other type of variational calculus and we see this as an alternative perspective for studying similar problems. Instead, we employ control theory techniques to, shortly speaking, reduce the overall problem to a critical case for which many results in the literature are available. In particular, our technique entitles us to find a \emph{worst case} SDE whose solution has an explicit density that dominates all densities of solutions to SDEs among those with measurable bounded drifts.

We believe this method is robust since no well-behaviour on the drift is needed other than merely boundedness and no Markovianity of the system is assumed. Certainly, no regularity is obtained but we are confident that the method can be exploited to gain more regularity of the densities.

This paper is organised as follows. In Section \ref{main results} we summarise our main results with some generalisations to non-trivial diffusion coefficients and to any arbitrary dimension. We also give some insight on concrete properties of the bounds as well as some examples with graphics. Section \ref{s:reduction} is devoted to thoroughly prove the assertions of the main results. More specifically, we will give an argument based on a control problem to reduce the problem to one critical case. 
We will also prove in detail the properties adduced in the previous section. 

\subsection{Notations}
We denote the strictly positive numbers by $\mathbb R_{++}:=(0,\infty)$, the trace of a matrix $M\in\mathbb R^{d\times d}$ by $\tr(M):=\sum_{j=1}^d M_{j,j}$ and $\pm$ simply denotes either $+$ or $-$. The Skorokhod space $\skor$ is the set of all c\`{a}dl\`{a}g functions from $\mathbb R_+$ to $\mathbb R^d$ equipped with the Skorokhod metric, c.f. \cite[Chapter VI.1]{js.87}. The canonical space is the triplet $\canspace$ where $\mathfrak D$ is the $\sigma$-algebra generated by the point evaluations and $(\mathcal D_t)_{t\geq 0}$ is the right-continuous filtration generated by the canonical process $X:\mathbb R_+\times \skor \rightarrow\mathbb R^d,(t,f)\mapsto f(t)$. 
We denote the generalised signum function by $\sgn(x):=1_{\{x\neq 0\}}x/\vert x\vert$ for any $x\in\mathbb R^d$. This is the orthogonal projection to the unit Euclidean sphere. For a complex number $z\in\mathbb C$ we denote its real resp.\ imaginary part by $\Re(z)$ resp.\ $\Im(z)$.

Further notations are used as in \cite{js.87}.

\section{Main results}\label{main results}

In this section we present our main result and some direct consequences. In particular, we will find sharp explicit bounds for SDEs with additive noise in the one-dimensional case and give some extensions to the $d$-dimensional case with more general diffusion coefficients.

Throughout this section let $\fraum$ be a filtered probability space with the usual assumptions on the filtration $\mathcal{F}=(\mathcal{F}_t)_{t\geq 0}$, i.e.\ $\mathcal F_0$ contains all $P$-null sets and $\mathcal F$ is right-continuous, $W$ be a $d$-dimensional standard Brownian motion and we define the process classes
\begin{align*}
 \mathcal A_+ &:= \{u: u\text{ is a stochastic process bounded by }1\}\\
 \mathcal A &:= \{ u\in\mathcal A_+: u\text{ is }\F\text{-adapted}\}.
\end{align*}

The next results constitutes one of the core results of this section and will be proven in detail in the next section.

\begin{thm}\label{t:main theorem, multi dimensional}
 Let $C>0$, $W$ be a $d$-dimensional standard Brownian motion and $u\in\mathcal A$. Then $X(t):=\int_0^tCu(s)ds + W(t)$ has Lebesgue density
 $$\rho_t(x) := \limsup_{\epsilon\rightarrow 0}\frac{P(\vert X(t)-x\vert \leq \epsilon)}{V_{\epsilon}},\ x\in\mathbb R^d$$
where $V_\epsilon = \frac{\pi^{d/2}}{\Gamma(d/2+1)}\epsilon^d$ denotes the volume of the $d$-dimensional Euclidean ball with \mbox{radius $\epsilon$} and $\Gamma$ denotes the gamma function. Moreover, $\rho_t$ satisfies
$$0< \alpha_{d,t,C}(x) \leq  \rho_t(x) \leq \beta_{d,t,C}(x) \leq \beta_{d,t,C}(0) $$
for any $t>0$, $x\in\mathbb R$ where 
  $$\alpha_{d,t,C}(x) := \limsup_{\epsilon\rightarrow 0}\frac{P(\vert Y_{Cx}^{+}(tC^2)\vert \leq C\epsilon)}{V_\epsilon}, \ \ \beta_{d,t,C}(x) := \limsup_{\epsilon\rightarrow 0}\frac{P(\vert Y_{Cx}^{-}(tC^2)\vert \leq C\epsilon)}{V_\epsilon},$$
   and $Y_x^{+}$ and $Y_x^{-}$ are the unique solutions to the SDEs
 \begin{align*}
  Y_x^{+}(t) &= x + \int_0^t \sgn(Y_x^{+}(s))ds + W(t), \\
  Y_x^{-}(t) &= x - \int_0^t \sgn(Y_x^{-}(s))ds + W(t)
 \end{align*}
 for any $t\geq 0$.
\end{thm}
\begin{proof}
 See at the end of Section \ref{s:reduction}.
\end{proof}

If $d=1$, then the functions $\alpha$, $\beta$ as well as some of their properties can be derived explicitly, cf.\ Theorem \ref{t:alphabeta}. In the multidimensional case we can give some of their properties. Let us summarise the formulas.
\begin{thm}\label{summarythm}
 Let $t>0$, $C>0$ and $\alpha$, $\beta$ be given as in Theorem \ref{t:main theorem, multi dimensional}. Then
\begin{align*}
 \alpha_{1,t,C}(0) &= \frac{1}{\sqrt{t}}\varphi\left(C\sqrt{t}\right)-C\Phi\left(-C\sqrt{t}\right),\quad\text{and} \\
 \beta_{1,t,C}(0) &= \frac{1}{\sqrt{t}}\varphi\left(C\sqrt{t}\right)+C\Phi\left(C\sqrt{t}\right)
\end{align*}
where $\Phi$ resp.\ $\phi$ denotes the distribution resp.\ density function of the standard normal law. For $x\in\mathbb R\backslash\{0\}$ we have
\begin{align*}
  \alpha_{1,t,C}(x) &=\int_0^{tC^2} C\alpha_{1,tC^2-s,1}(0)\rho_{\theta_0^{Cx}}(s)ds\quad\text{and}\\
  \beta_{1,t,C}(x) &=\int_0^{tC^2}   C\beta_{1,tC^2 -s,1}(0) \rho_{\tau_0^{Cx}}(s)ds
\end{align*}
where 
\begin{align*}
 \rho_{\tau_0^x}(t)   &= \frac{|x|}{\sqrt{2\pi t^3}}e^{-\frac{(|x|-t)^2}{2t}}\quad{and}\\
 \rho_{\theta_0^x}(t) &= \frac{|x|}{\sqrt{2\pi t^3}}e^{-\frac{(|x|+t)^2}{2t}}
\end{align*}
for any $s>0$. Moreover, we have
$$\frac{2^d}{C_d d^{d/2}}\prod_{i=1}^d\alpha_{1,t,C}(x_i) \leq \alpha_{d,t,C}(x)\leq \beta_{d,t,C}(x)\leq \frac{2^d}{C_d}\prod_{i=1}^d\beta_{1,t,C}(x_i), \quad x\in \mathbb{R}^d$$
where $C_d:= \frac{\pi^{d/2}}{\Gamma\left( \frac{d}{2}+1\right)}$ for any $x\in\mathbb R^d$.
\end{thm}
\begin{proof}
 This is part of the statements of Theorems \ref{t:alphabeta} and \ref{t:multibounds} below.
\end{proof}

In what follows, we will derive bounds for the densities of solutions of general SDEs. The following is an immediate consequence of Theorem \ref{t:main theorem, multi dimensional}.
\begin{cor}\label{k:bound for SDE with additive noise}
  Let $C>0$, $x_0\in\mathbb R^d$, $b:\mathbb R_+\times C(\mathbb R_+,\mathbb R^d)\rightarrow \mathbb R$ be predictable and bounded by $C$. Then any weak solution of the SDE 
 $$ X(t) = x_0 + \int_0^t b(s,X) ds + W(t), \quad t\geq 0 $$
  has density $\rho_t$ at time $t>0$ which is bounded from below by $x\mapsto \alpha_{d,t,C}(x-x_0)$ and from above by $x\mapsto \beta_{d,t,C}(x-x_0)$ where $\alpha$ and $\beta$ are given in Theorem \ref{t:main theorem, multi dimensional} and $W$ is a $d$-dimensional Brownian motion. Moreover, the bounds are optimal in the sense that for any $x_1,x_2\in\mathbb R^d$ there are two functionals $b_{x_1}$, resp. $b_{x_2}$ for which the density $\rho_t$ of the solution to the SDE $dX(t)=b_{x_1}(X(t))dt + W(t), X(0)=0$, resp.\ $dX(t)=b_{x_2}(X(t))dt + W(t), X(0)=0$ attains the upper bound in $x_1$, resp. the lower bound in $x_2$.
\end{cor}
\begin{proof}
 Define $Y(t) := X(t) - x_0$ and $u(t) := b(t,X)$ for any $t\geq 0$. Then
 $$ Y(t) = \int_0^t u(s) ds + W(t),\quad t\geq 0.$$
 The bounds follow from Theorem \ref{t:main theorem, multi dimensional}. Shifts of the processes $Y^-$, resp.\ $Y^+$ attain the upper, resp.\ lower bounds at the given points.
\end{proof}

Now we focus on our second main result which is an application of Corollary \ref{k:bound for SDE with additive noise}. This time $X$ is given as a solution of an SDE with measurable drift and a diffusion coefficient which is continuously differentiable.
\begin{thm}\label{t:generelle SDE}
 Let $b:\mathbb R_+\times C(\mathbb R_+,\mathbb R^d)\rightarrow \mathbb R^d$ be predictable, $\sigma:\mathbb R_+\times \mathbb R^d\rightarrow \mathbb R^{d\times d}$ be continuously differentiable and assume the following conditions.
 \begin{enumerate}
   \item $\sigma(t,x)$ is an invertible matrix for any $t\geq0$, $x\in\mathbb R^d$.
   \item There is a function $F:\mathbb R_+\times\mathbb R^d\rightarrow \mathbb R^d$ such that $D_2F(t,x) =(\sigma(t,x))^{-1}$ for any $t\geq 0$, $x\in\mathbb R^d$ where $D_2F(t,x)$ denotes the Fr\'echet derivative of $F(t,\cdot)$ with respect to $x$.
   \item The function
  \begin{align*}
   \widetilde b&:\mathbb R_+\times C(\mathbb R_+,\mathbb R^d)\rightarrow \mathbb R^d,\\
   (t,f)&\mapsto \partial_1F(t,f(t)) + \sigma(t,f(t))^{-1}b(t,f) \\
      &\quad+\frac{1}{2}\left(\tr\big(\ \sigma(t,f(t))^\top H_2F_k(t,f(t))\sigma(t,f(t))\ \big)\right)_{k=1,\dots,d}
  \end{align*}
   is bounded by some constant $C>0$ where $H_2F_k(t,x)$ denotes the Hessian matrix of $F_k(t,\cdot)$, i.e.\ $(\partial_{x_i}\partial_{x_j}F_k(t,x))_{i,j=1,\dots,d}$ for any $t\geq 0$, $x\in\mathbb R^d$.
 \end{enumerate}
  Then any solution of the SDE
  $$ X(t) = x_0 + \int_0^t b(s,X)ds + \int_0^t \sigma(s,X(s))dW(s)$$
  has, at each time $t$, Lebegsue density $\rho_t$ and for every $x\in\mathbb R^d$ we have
  $$ \rho_t(x) \leq \frac{\beta_{d,t,C}(F(t,x)-F(0,x_0))}{\tr(\sigma(t,x))} $$
  where $\alpha_{d,t,C}$, $\beta_{d,t,C}$ are defined as in Theorem \ref{t:main theorem, multi dimensional}. Moreover, if additionally $F(t,\cdot)$ is invertible for any fixed $t>0$, then
  $$ 0<\frac{\alpha_{d,t,C}(F(t,x)-F(0,x_0))}{\tr(\sigma(t,x))} \leq \rho_t(x) \leq \frac{\beta_{d,t,C}(F(t,x)-F(0,x_0))}{\tr(\sigma(t,x))}. $$
\end{thm}
\begin{proof}
 Define $Y(t):= F(t,X(t))$ and $u(t):= \widetilde b(t,X)$ for any $t\geq 0$. Then It\^o's formula yields
 $$ Y(t) = F(0,x_0) + \int_0^t u(s) ds + W(t),\quad t\geq 0. $$
 Theorem \ref{t:main theorem, multi dimensional} states that $Y(t)$ has Lebesgue density $\rho_{Y(t)}$ which admits the bounds
 $$ \alpha_{d,t,C}(y-F(0,x_0)) \leq \rho_{Y(t)}(y) \leq \beta_{d,t,C}(y-F(0,x_0))$$
 for any $t> 0$, $y\in\mathbb R^d$.

 From the definition of $Y(t)$ we directly get
 $$ \rho_{t}(x) \leq \frac{\rho_{Y(t)}(F(t,x)-F(0,x_0))}{\tr(\sigma(t,x))}\leq \frac{\beta_{d,t,C}(F(t,x)-F(0,x_0))}{\tr(\sigma(t,x))} $$
 for any $t>0$, $x\in\mathbb R^d$.

 If we assume that $F(t,\cdot)$ is invertible for any $t>0$, then
 $$ \rho_{t}(x) = \frac{\rho_{Y(t)}(F(t,x)-F(0,x_0))}{\tr(\sigma(t,x))} $$
 for any $x\in\mathbb R^d$ and, hence, the additional claim follows.
\end{proof}

The conditions (1) to (3) appearing in Theorem \ref{t:generelle SDE} simplify considerably in dimension $1$. Moreover, due to It\^o-Tanaka's formula we can relax the conditions on $\sigma$.
\begin{thm}
 Let $X$ be a solution of the SDE 
  $$X(t) = x_0+ \int_0^tb(s,X)dt + \int_0^t\sigma(X(s))dW(s)$$
 where $x_0\in\mathbb R$, $W$ is a standard Brownian motion, $b:\mathbb R_+\times C(\mathbb R_+,\mathbb R)\rightarrow \mathbb R$ predictable and bounded by some constant $C_b$, $\sigma:\mathbb R\rightarrow\mathbb R_{+}$ is a Lipschitz continuous function with Lipschitz bound $L$ and $\sigma(x)\geq \epsilon$ for some constant $\epsilon>0$.

 Then $X(t)$ has Lebesgue density $\rho_t$ and
  $$0<\frac{\alpha_{t,C}(\vert F(x)-F(x_0)\vert)}{\sigma(x)}\leq \rho_t(x) \leq \frac{\beta_{t,C}(\vert F(x)-F(x_0)\vert)}{\sigma(x)}$$
 for any $t>0$ where $\alpha_{t,C}$ and $\beta_{t,C}$ are defined as in Theorem \ref{t:main theorem, multi dimensional} when $d=1$, $F(x):=\int_0^x \frac{1}{\sigma(u)}du$ and 
$$C:=\sup\left\{\left\vert\frac{b(t,f)}{\sigma(f(t))}\right\vert:t\in\mathbb R_+,f\in C(\mathbb R_+,\mathbb R)\right\}+L/2.$$
 Moreover, $C\leq \frac{C_b}{\epsilon}+L/2$ where $C_b$ is a uniform bound for $b$.
\end{thm}
\begin{proof}
 Define $Y(t):=F(X(t))$. Since $\sigma$ is Lipschitz continuous there is a function $\sigma':\mathbb R_+\rightarrow\mathbb R$ which is bounded by $L$ and $\sigma(x) = \sigma(0) + \int_0^x \sigma'(u)du$. Then It\^o-Tanaka's formula \cite[Theorem VI.1.5]{revuz.yor.99} yields
 $$ Y(t) = F(x_0) + \int_0^t\left(\frac{b(s,X)}{\sigma(X(s))}-\frac{1}{2}\sigma'(X(s))\right)ds + W(t).$$
 Let $G:=F^{-1}$ and define
 $$\widetilde b(s,y) := \frac{b(s,G\circ f)}{\sigma(G(f(s))} - \frac{1}{2}\sigma'(G(f(s))),\quad s\in\mathbb R_+,f\in C(\mathbb R_+,\mathbb R)$$
 which is predictable and bounded by $C$. Then the result follows from Corollary \ref{k:bound for SDE with additive noise}.
\end{proof}

In the next section we will give precise definitions and mathematical computations of the functions $\alpha_{d,t,C}$ and $\beta_{d,t,C}$ in dimension 1 and why these are the optimal bounds (in the sense of Corollary \ref{k:bound for SDE with additive noise}) for the densities of SDEs with bounded measurable drifts. Before we do that, let us give some intuitive insight on the shape and behaviour of these bounds for the one-dimensional case. Consider any one-dimensional process of the form 
$$X(t) = \int_0^t u(s)ds + W(t),\quad  t\geq 0, \quad u\in \mathcal{A}$$ 
as in Theorem \ref{t:main theorem, multi dimensional}. In particular, $X$ can be the solution to the following SDE, $dX(t) = b(t,X)dt+dW(t)$, $X(0)=0$, $t\geq 0$, with $b$ bounded and predictable as in Corollary \ref{k:bound for SDE with additive noise}. Furthermore, denote by $\rho_t$ the density of $X(t)$ at a fixed time $t>0$. Then Theorem \ref{t:main theorem, multi dimensional} grants that $0<\alpha_t(x) \leq \rho_t(x) \leq \beta_t(x)$ for any $x\in \mathbb{R}$. In the following figure we can observe the functions $\alpha_t$ and $\beta_t$ for different values of $t>0$ and see how they behave. We can see the function $\alpha_t$ in orange and $\beta_t$ in green. Any density lies between these two curves and these bounds are optimal in the sense that, for given $x_0,y_0\in \mathbb{R}$ we can find drifts $u_{x_0}$ and $u_{y_0}$ such that the associated densities $\rho_t^{x_0}$, resp. $\rho_t^{y_0}$ for these drift coefficients satisfy $\rho_t(x_0)=\alpha_t(x_0)$, respectively, $\rho_t(y_0)=\beta_t(y_0)$. As an illustration we just take the drift to be $+\sgn(x-0.25)$ in blue and $-\sgn(x-1)$ in red.

 \begin{figure}[H]
 \centering
 \includegraphics[scale=0.5, angle=270]{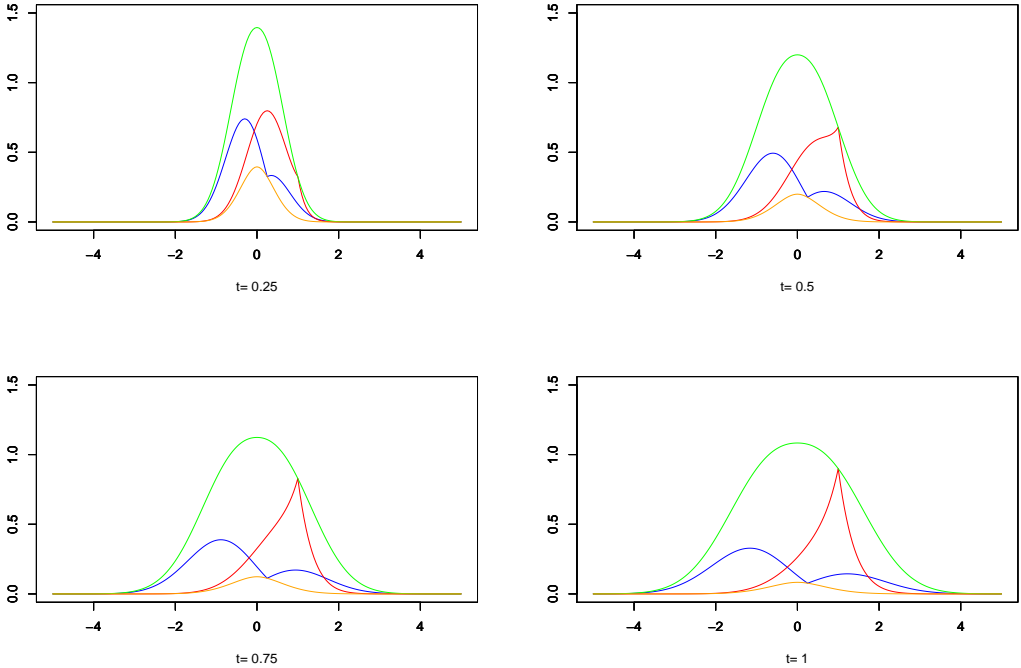}
 \caption{Upper and lower bounds for $C=1$ starting at $x=0$ (in green and orange) with the respective densities when the drift coefficients are $\sgn(x-0.25)$ and $-\sgn(x-1)$ (blue and red) at different times \mbox{$t\in\{0.25, 0.5, 0.75, 1\}$}.}
 \end{figure}

As we can see, both densities are bounded by $\alpha_t$ and $\beta_t$ and the bounds are attained in $0.25$ for density of the process with drift $+\sgn(x-0.25)$ (in blue) and in $1$ when the drift is $-\sgn(x-1)$ (in red).

\section{Reduction and the critical case}\label{s:reduction}

In this section we will see how to derive the functions $\alpha_{t,C}$ and $\beta_{t,C}$ explicitly for the case $d=1$ as well as some of their properties, cf.\ Theorem \ref{t:alphabeta}. Then we will show that these are indeed the bounds for the densities of any solution to SDEs with bounded measurable drift by solving a stochastic control problem, cf.\ Theorem \ref{t:control problem} and thereafter we give the proof for Theorem \ref{t:main theorem, multi dimensional}. In the sequel, consider the process

\begin{align}\label{e:signum SDE}
Y_x^{\pm}(t) := x \pm \int_0^t \sgn (Y_x^{\pm}(s))ds + W(t), \quad t\geq 0,
\end{align}
c.f.\ \cite{veretennikov.79} for existence and (pathwise) uniqueness. Moreover, at some point we will also use the property that the solution to equation (\ref{e:signum SDE}) is strong Markov, even for the multidimensional case. This can be for instance justified using \cite[Theorem 6.4.5]{applebaum.09} in connection with \cite[Corollary IX.1.14]{revuz.yor.99}.


\begin{lem}\label{0:densities}
For every $t>0$, $Y_0^{+}(t)$ resp.\ $Y_0^{-}(t)$ has density $\rho_{Y_0^{+}(t)}$, resp.\ $\rho_{Y_0^{-}(t)}$ given by
 \begin{align*}
p_{t}(0,y) &:= \rho_{Y_0^{+}(t)} = \frac{1}{\sqrt{t}}\phi\left(\frac{\vert y\vert -t}{\sqrt{t}}\right)-e^{2|y|}\Phi\left(-\frac{\vert y\vert + t}{\sqrt{t}}\right),\quad\text{resp.} \\
 q_{t}(0,y) &:= \rho_{Y_0^{-}(t)} = \frac{1}{\sqrt{t}}\phi\left(\frac{t+\vert y\vert}{\sqrt{t}}\right)+e^{-2\vert y\vert}\Phi\left(\frac{t-\vert y\vert}{\sqrt{t}}\right)
 \end{align*}
 for $y\in\mathbb R$ and any $t>0$ where $\phi$, resp.\ $\Phi$, denote the density, resp.\ the distribution function, of the standard normal law.
\end{lem}
\begin{proof}
 The density for $Y_0^{-}(t)$ is the statement of \cite[Exercise 6.3.5]{karatzas.shreve.91} as for $Y_0^{+}(t)$ computations are fairly similar.
\end{proof}

The computation of the densities $\rho_{Y_0^{+}(t)}$ and $\rho_{Y_0^{-}(t)}$ in the previous lemma are relatively easy given the fact that the local-time of the Brownian motion starting from 0 is symmetric and the joint law of $W(t)$ and the local time of $W$, $L_t^W (0)$ is explicitly known, see \cite{karatzas.shreve.91}. Nevertheless, one is able to find reasonably explicit expressions for the densities of $Y_x^{+}(t)$ and $Y_x^{-}(t)$ which yield representations for $\alpha$ and $\beta$ if $d=1$.

First we focus on the computation of the density of $Y_x^{-}(t)$ and then for $Y_x^{+}(t)$ which is similar.

\begin{lem}\label{l:density minus}
For every $t\geq 0$, the density of $Y_x^{-}(t)$ is given by
$$q_{t}(x,y) =\frac{1}{\sqrt{2\pi t}}e^{-\frac{(\sgn(x)(x-y)-t)^2}{2t}}\left( 1- e^{-\frac{2xy}{t}}\right)1_{\{\sgn (xy)\geq 0\}} + \int_0^t   q_{t-s}(0,y) \rho_{\tau_0^x}(s)ds$$
where $x,y\in \mathbb{R}$, $x\neq 0$ and $\tau_0^x$ is the first hitting time of the process $Y_x^{-}(t)$ at 0 whose density function is explicitly given by
$$\rho_{\tau_0^x}(s) = \frac{|x|}{\sqrt{2\pi s^3}}e^{-\frac{(|x|-s)^2}{2s}}, \ s>0.$$
\end{lem}
\begin{proof}
Let $\tau_0^x$ be the first time the process $Y_x^{-}$ hits 0, i.e.
$$\tau_0^x := \inf \{t\geq 0: \ Y_x^{-}(t) = 0\}.$$
Then it is clear, that $Y_x^{-}(t) = x - \sgn (x)t+W(t)$ for any $t\in[0,\tau_0^x]$. Define $\widetilde{W}:=-W$ and $B(t):= \sgn(x)t + \widetilde{W}(t)$. The process $B(t)$ is a Brownian motion with drift starting at 0. It is clear, that $\tau_0^x =\inf \{t\geq 0: B(t) =x\}$, whose law is known, namely $\tau_0^x$ is inverse Gaussian distributed and \cite[p.223, Formula 2.0.2]{borodin.salminen.96} states that its density is given by
$$\rho_{\tau_0^x}(t) = \frac{|x|}{\sqrt{2\pi t^3}}e^{-\frac{(|x|-t)^2}{2t}}, \ t>0.$$

Now define $f_{\varepsilon}(z) := \frac{1}{2\varepsilon} 1_{(y-\varepsilon, y+\varepsilon)}(z)$ for a fixed $y\in \mathbb{R}$, then
\begin{align*}
\E[f_{\varepsilon} (Y_x(t))] &= \E[f_{\varepsilon} (Y_x^{-}(t))1_{\{t<\tau_0^x\}}] + \E[f_{\varepsilon} (Y_x^{-}(t))1_{\{t\geq\tau_0^x\}}]\\
&= A_1+A_2
\end{align*}
where $A_1:=\E[f_{\varepsilon} (Y_x^{-}(t))1_{\{t<\tau_0^x\}}]$ and $A_2:=\E[f_{\varepsilon} (Y_x^{-}(t))1_{\{t\geq\tau_0^x\}}]$.
We have
\begin{align*}
P\left(Y_x^{-}(t) \leq y, t< \tau_0^x\right) &= P\left(x-\sgn (x) t + W(t) \leq y,t< \tau_0^x\right)\\
&=P\left(B(t)\geq x-y,t< \tau_0^x\right).
\end{align*}

We start with the case $x>0$. Observe that $\tau_0^x = \inf \{t>0: B(t)=x\}$ and hence $\{t < \tau_0^x\}=\{M(t) < x\}$ where $M(t):=\sup_{s\in [0,t]} B(s)$. As a consequence
\begin{align*}
P\left(Y_x^{-}(t) \leq y, t< \tau_0^x\right) &= P\left(B(t) \geq x-y,  M(t)< x\right)\\
&= \E\left[1_{\{B(t)\geq x-y,  M(t)< x\}}\right]\\
&= \E_Q \left[1_{\{B(t)\geq x-y, M(t)< x\}}\frac{1}{Z(t)}\right]
\end{align*}
where $Q$ is the equivalent measure w.r.t. $P$ defined by
$$\frac{dQ}{dP}\bigg|_{\mathcal{F}_t} =\exp\left\{-\sgn(x)\widetilde W(t)- t/2\right\}=: Z(t), \quad t\geq 0.$$
\cite[Theorem 8.6.4]{oeksendal.03} yields that the process $B(t) = \sgn(x)t + \widetilde W(t)$, $t\geq 0$ is a standard $Q$-Brownian motion and $M(t)$ is therefore the running maximum of the standard Brownian motion $B$, hence
\begin{align}\label{PYM}
P\left(Y_x^{-}(t) \leq y, t\leq\tau_0^x\right) = \int_0^{\infty}\int_{-\infty}^w 1_{\{z\geq x-y,  w< x\}}e^{\sgn(x) z - t/2}\rho_{B(t),M(t)}(z,w)dzdw
\end{align}
where $\rho_{B(t),M(t)}$ denotes the joint density of $B(t)$ and $M(t)$ which is explicitly given, see \mbox{\cite[Proposition 2.8.1]{karatzas.shreve.91}}, by
$$\rho_{B(t),M(t)}(z,w)=\frac{2(2w-z)}{\sqrt{2\pi t^3}}\exp\left\{ -\frac{(2w-z)^2}{2t}\right\}, \ z \leq w, \ w\geq 0.$$
We have
\begin{align*}
A_1 &= \frac{1}{2\varepsilon}P\left( y-\varepsilon \leq Y_x^{-}(t) \leq y+\varepsilon, t\leq \tau_0^x \right)\\
&=\frac{1}{2\varepsilon}\int_0^{\infty}\int_{-\infty}^w 1_{\{x-y-\varepsilon \leq z\leq x-y+\varepsilon,  w< x\}} e^{\sgn(x) z - t/2}\rho_{B(t),M(t)}(z,w)dzdw
\end{align*}
Finally, the above probability converges to the derivative of (\ref{PYM}) w.r.t. $y$, that is
\begin{align*}
\lim_{\varepsilon \searrow 0}&\frac{1}{2\varepsilon}P\left( y-\varepsilon \leq Y_x^{-}(t) \leq y+\varepsilon,  t< \tau_0^x \right)\\
&= e^{\sgn(x) (x-y) - t/2} \int_{x-y}^{x} \rho_{B(t),M(t)}(x-y,w)dw\\
&= \frac{1}{\sqrt{2\pi t}}e^{\sgn(x) (x-y) - t/2} \left(e^{-(x-y)^2/2t} - e^{-(x+y)^2/2t}\right)1_{\{x\geq x-y\}}\\
&= \frac{1}{\sqrt{2\pi t}}e^{-\frac{(\sgn(x)(x-y)-t)^2}{2t}}\left( 1- e^{-\frac{2xy}{t}}\right)1_{\{y\geq 0\}}. 
\end{align*}

Now we continue to compute $A_2$. Define the random variable $\tau:= \tau_0^x \vee t$. It is readily checked that $\tau \geq \tau_0^x$ and $\tau$ is $\mathcal{F}_{\tau_0^x}$-measurable because the event $\{t\geq\tau_0^x\}$ is in $\mathcal{F}_{\tau_0^x}$. Then the strong Markov property of $Y_x^{-}$ and \mbox{\cite[Corollary 2.6.18]{karatzas.shreve.91}} yield
\begin{align*}
  \E[f_{\varepsilon} (Y_x^{-}(t))1_{\{t\geq\tau_0^x\}}|\mathcal{F}_{\tau_0^x}] &= \E[f_{\varepsilon} (Y_x^{-}(\tau))1_{\{t\geq\tau_0^x\}}|\mathcal{F}_{\tau_0^x}] \\
  &= 1_{\{t\geq\tau_0^x\}}\E[f_{\varepsilon} (Y_x^{-}(\tau))|\mathcal{F}_{\tau_0^x}] \\
  &= 1_{\{t\geq\tau_0^x\}}\E[f_{\varepsilon}(Y_0^{-}(\xi))]|_{\xi = \tau-\tau_0^x}
\end{align*}
$P$-a.s. As a consequence
\begin{align*}
\E[f_{\varepsilon} (Y_x^{-}(t))1_{\{t\geq\tau_0^x\}}] &= \E\left[\E[f_{\varepsilon} (Y_x^{-}(t))1_{\{t\geq\tau_0^x\}}|\mathcal{F}_{\tau_0^x}] \right]\\
&=\E\left[1_{\{t\geq\tau_0^x\}} \E[f_{\varepsilon}(Y_0^{-}(\xi))]|_{\xi = \tau-\tau_0^x} \right]\\
&=\E\left[1_{\{t\geq\tau_0^x\}} \E[f_{\varepsilon}(Y_0^{-}(\xi))]|_{\xi = t-\tau_0^x} \right].
\end{align*}
Now, the density of $Y_0^-(t)$ is explicitly known by Lemma \ref{0:densities}. Thus
$$A_2 =\E\left[\int_{\mathbb{R}} f_{\varepsilon}(z) q_{t-\tau_0^x}(0,z)1_{\{t\geq \tau_0^x\}}\right] = \int_0^t \int_{\mathbb{R}} f_{\varepsilon}(z) q_{t-s}(0,z)\rho_{\tau_0^x}(s)ds.$$

Then, letting $\varepsilon \to 0$ and by Lebesgue's dominated convergence theorem we obtain that, for $x>0$ and $y\in \mathbb{R}$
$$q_{t}(x,y) =\frac{1}{\sqrt{2\pi t}}e^{-\frac{(\sgn(x)(x-y)-t)^2}{2t}}\left( 1- e^{-\frac{2xy}{t}}\right)1_{\{y\geq 0\}} + \int_0^t   q_{t-s}(0,y) \rho_{\tau_0^x}(s)ds.$$

We have
\begin{align*}
  -Y_{-x}^-(t) &= x + \int_0^t \sgn(Y_{-x}^-(s)) ds + \widetilde W(t) \\
             &= x - \int_0^t \sgn(-Y_{-x}^-(s)) ds + \widetilde W(t)
\end{align*}
for any $t\geq 0$ and hence $(-Y_{-x}^-,\widetilde W)$ is a weak solution of \eqref{e:signum SDE} for $\pm=-$ and starting point $x$. Hence, $-Y_{-x}^-(t)$ has the same law as $Y_{x}^-(t)$ for any $t\geq 0$. Consequently, we have
 $$ q_t(x,y) = q_t(-y,-x),\quad x>0,y\in\mathbb R.$$
The claimed formula follows.
\end{proof}

Similarly, we can also obtain the density for $Y_x^{+}(t)$. The proof follows exactly the same ideas as in Lemma \ref{l:density minus} and has therefore been omitted.

\begin{lem}\label{l:density plus}
For every $t\geq 0$, the density of $Y_x^{+}(t)$ is given by
$$p_{t}(x,y):=\frac{2}{\sqrt{2\pi t}}e^{-\frac{(\sgn(x)(x-y)+t)^2}{2t}}\left(1-e^{\frac{-2xy}{t}}\right)1_{\{\sgn(xy)\geq 0\}} + \int_0^t p_{t-s}(0,y)\rho_{\theta_0^x}(s)ds.$$
for $x,y\in \mathbb{R}$, $x\neq 0$ and $\theta_0^x$ is the first hitting time of where
$$\rho_{\theta_0^x}(s)=\frac{|x|}{\sqrt{2\pi s^3}}e^{-\frac{(|x|+s)^2}{2s}}, \ \ 0<s<\infty.$$
\end{lem}
\begin{proof}
The proof of this lemma follows completely the same ideas as in Lemma \ref{l:density minus}. One of the main differences is that in this case the distribution of the stopping time $\theta_0^x$ has an atom at infinity, namely, from \cite[p.223, Formula 2.0.2]{borodin.salminen.96} we have

$$\rho_{\theta_0^x} (t) = \frac{|x|}{\sqrt{2\pi t^3}} e^{- \frac{(|x|+t)^2}{2t}}, \ 0<t<\infty$$
and
$$P(\theta_0^x = \infty) = 1- e^{-2|x|}.$$

\end{proof}

Now we are in a position to define the functions $\alpha_{t,C}$ and $\beta_{t,C}$ for the one-dimensional case and study some of their properties. Before we do that, we will need a technical result to prove one of the properties of these functions.

\begin{prop}\label{p:stochineq}
Let $b:\mathbb R_+\times \mathbb R\rightarrow\mathbb R$ be bounded and measurable and
   $$ X_x(t) := x + \int_0^t b(s,X_x(s)) ds + W(t),\quad x\in\mathbb R, \quad t\geq 0$$
where $W$ is a $1$-dimensional Brownian motion. Then
    $$ X_x(t)\leq X_y(t)\quad P\text{-a.s.}$$
    for any $t\geq 0$, $x,y\in\mathbb R$ with $x\leq y$.
 \end{prop}
 \begin{proof}
   Define
   $$ Y_x(t) := X_x(t) - W(t) = x + \int_0^t b(s,Y_x(s)+W(s))ds = x+\int_0^t\widetilde{b}(s,Y_x(s))ds$$
   where the equalities hold $P$-a.s. and here $\widetilde{b}(t,z):= b(t,z+W(t))$ for any $t\geq 0$, $z\in\mathbb R$. Let $x,y\in\mathbb R$ with $x\leq y$ and define $Z(t):=\min\{Y_x(t),Y_y(t)\}$. Then
    $$ Z(t) = x + \int_0^t \widetilde{b}(s,Z(s)) ds,\quad t\geq 0.$$
  Hence $U(t) := Z(t) + W(t) = x + \int_0^t b(s,U(s)) ds + W(t)$. \cite[Theorem IX.3.5 i)]{revuz.yor.99} yields $U(t) = X_x(t)$ a.s. Observe that $U(t) = \min\{X_x(t),X_y(t)\}$ and hence
    $$ X_x(t) = U(t) \leq X_y(t),\quad t\geq 0$$
    $P$-a.s.
\end{proof}

\begin{thm}\label{t:alphabeta}
Let $q$ be the transition density of the Markov process $Y^-$ which is given in Lemma \ref{l:density minus} and $p$ the transition density for the Markov process $Y^+$ given in \mbox{Lemma \ref{l:density plus}}. Define the functions $\alpha,\beta:\mathbb{R}_{++}\times \mathbb{R}_+ \times \mathbb{R}\rightarrow (0,\infty)$ by \mbox{$\alpha_{t,C}(x):=C p_{tC^2}(Cx,0)$} and \mbox{$\beta_{t,C}(x):=C q_{tC^2}(Cx,0)$} where $t>0$, $C>0$ and $x\in\mathbb R$. Then
\begin{align}\label{alpha}
\begin{split}
\alpha_{t,C}(x)&=\int_0^{tC^2} Cp_{tC^2-s}(0,0)\rho_{\theta_0^{Cx}}(s)ds,\\
&= \int_0^{tC^2} \left(\frac{C}{\sqrt{tC^2-s}}\varphi(\sqrt{tC^2-s}) - C\Phi(-\sqrt{tC^2-s}) \right)\rho_{\theta_0^{Cx}}(s)ds, \ x\neq 0,
\end{split}
\end{align}
and
\begin{align}\label{beta}
\begin{split}
\beta_{t,C}(x) &=\int_0^{tC^2}   Cq_{tC^2 -s}(0,0) \rho_{\tau_0^{Cx}}(s)ds \\
 &= \int_0^{tC^2} \left(\frac{C}{\sqrt{tC^2-s}}\phi(\sqrt{tC^2-s})+C\Phi(\sqrt{tC^2-s})\right) \rho_{\tau_0^{Cx}}(s)ds,\quad x\neq 0
 \end{split}
\end{align}
where recall that $\rho_{\theta_0^x}$, respectively $\rho_{\tau_0^x}$ are given as in Lemma \ref{l:density plus}, respectively as in Lemma \ref{l:density minus}.

In addition, for each $t>0$ and $C> 0$ the functions $\alpha_{t,C}$ and $\beta_{t,C}$ are analytic in $\mathbb{R}\setminus\{0\}$, Lipschitz continuous in $\mathbb{R}$, symmetric, decreasing on $[0,\infty)$ and by symmetry increasing on $(-\infty,0]$. They have exponential decay of the type $o(c_1|x|e^{c_2|x|}e^{-c_3|x|^2})$ for constants $c_1,c_2, c_3>0$. Moreover, they attain their maxima at $x=0$ which are given by
$$\alpha_{t,C}(0)=Cp_{tC^2}(0,0)= \frac{1}{\sqrt{t}}\varphi\left( C\sqrt{t}\right) - C\Phi\left( -C\sqrt{t}\right)$$
and
$$\beta_{t,C}(0)=Cq_{tC^2}(0,0)=\frac{1}{\sqrt{t}}\varphi\left(C\sqrt{t}\right)+C\Phi\left(C\sqrt{t}\right).$$
\end{thm}
\begin{proof}
We will carry out a more detailed proof of the properties on $\beta_{t,C}$. For the case of $\alpha_{t,C}$ the same proof, \emph{mutatis mutandis}, follows as well.

First of all, observe that $\beta_{t,C}(x)=C \beta_{tC^2,1}(Cx)$ and hence it is sufficient to carry out the proof for $C=1$ then all properties follow for arbitrary $C>0$.

At the end of the proof of Lemma \ref{l:density minus} we have shown that the law of $Y_x^-(t)$ coincides with the law of $-Y_{-x}^-(t)$. Hence, the symmetry of $\beta_{t,1}$ follows.

To show analyticity, define \mbox{$f(s,x):=  q_{t-s}(0,0) \rho_{\tau_0^{x}}(s)$} for $s\in (0,t)$ and \mbox{$x\in \mathbb{R}\setminus \{0\}$} and the family of domains $$\mathbb{S}_{\varepsilon}:= \left\{z\in \mathbb{C}: \ \varepsilon < \Re(z) < \frac{1}{\varepsilon},\ \Re(z) > 2\vert \Im(z)\vert\right\},$$
 $0<\varepsilon <1$ and $\mathbb{S}:=\cup_{0<\varepsilon <1} \mathbb{S}_{\varepsilon}$. Then for every $z\in \mathbb{S}$, $g:\mathbb{R}_+\times \mathbb{S}\rightarrow \mathbb{C}$ defined as $g(s,z):=q_{t-s}(0,0)\frac{z}{\sqrt{2\pi s^3}}e^{-\frac{(z-s)^2}{2s}}$ is the holomorphic extension of $f$ to $\mathbb{S}$. Let $\epsilon>0$, $t>0$ and let us check that $z\mapsto \int_0^{t} g(s,z)ds$ is holomorphic on $\mathbb{S}_\varepsilon$. We have $\vert z\vert \leq \sqrt{5/4}/\epsilon$, $\Re(z^2)>3\epsilon^2/4$ and hence
\begin{align*}
 \vert g(s,z)\vert &\leq \left(\frac{1}{\sqrt{t-s}}+1\right)\frac{1/\varepsilon}{\sqrt{s^3}} \vert e^{-\frac{z^2}{2s}}e^{z}e^{-s/2}\vert \\
&\leq \left(\frac{1}{\sqrt{t-s}}+1\right)\frac{1/\varepsilon}{\sqrt{ s^3}}e^{1/\varepsilon}e^{-\frac{3\varepsilon^2}{8s}}
\end{align*}
for any $s\in(0,t)$, which is integrable on $(0,t)$ for every $\varepsilon>0$. For a real differentiable function from an open domain in $\mathbb C$ to $\mathbb C$ we denote the complex conjugate differential operator by $\partial_{\bar z}$. Recall, that such a function is holomorphic if and only if its complex conjugate derivative is zero. So, by changing differentiation and integration, we have
$$\partial_{\bar{z}} \int_0^{t} g(s,z)ds = \int_0^{t} \partial_{\bar{z}} g(s,z)ds=0$$
for every $z\in \mathbb{S}_{\varepsilon}$ where the last follows since $g(t,\cdot)$ is holomorphic on $\mathbb{S}$ for every $t>0$ being thus $\int_0^{t}f(s,x)ds$ is analytic on $(0,\infty)$. For $x<0$ use the symmetry of $\beta_{t,1}$ to conclude.

In addition, $\beta_{t,1}$ is Lipschitz in 0, i.e.\ there is a constant $K>0$ such that $\vert \beta_{t,1}(0) - \beta_{t,1}(x)\vert \leq \vert x\vert K$ for any $x\in\mathbb R$. Indeed, write
\begin{align*}
\int_0^{t}   q_{t -s}(0,0) \rho_{\tau_0^{x}}(s)ds
&=E[H(\tau_0^{x})]+ \int_{t/2}^{t} q_{t -s}(0,0) \rho_{\tau_0^{x}}(s)(1-h(s))ds
\end{align*}
where $H(s):=q_{t -s}(0,0)h(s)$ where $h$ is some function which is bounded by $1$, constant $1$ near zero, constant $0$ on $[t/2,t]$ and $h\in C^{\infty}([0,t],\mathbb R)$.

We see that $H$ is Lipschitz continuous with some Lipschitz constant $L>0$ and, hence,
$$\vert \E[H(\tau_0^{x})] - \E[H(\tau_0^0)]\vert \leq L (\E\tau_0^{x}-\E\tau_0^0) = L\vert x\vert$$
for any $x>0$. Moreover,
\begin{align}\label{lemmabeta1}
\int_{t/2}^{t} q_{t -s}(0,0) \rho_{\tau_0^{x}}(s)(1-h(s))ds \leq |x| \frac{1}{\sqrt{t}} \frac{2}{\pi}\int_{1/2}^1 \left( \frac{1}{\sqrt{2\pi t}} \frac{1}{\sqrt{1-s}}+1\right)ds
\end{align}
which implies that
$$\vert \beta_{t,1}(0) - \beta_{t,1}(x)\vert \leq \vert x\vert K$$
for some constant $K>0$. Together with the analyticity outside zero we conclude that $\beta_{t,1}$ is locally Lipschitz continuous. If we have shown that $\beta_{t,1}$ is decreasing on $[0,\infty)$, then it follows that $\beta_{t,1}$ is globally Lipschitz continuous because it is positive valued.

For monotonicity, it is sufficient to show that $\beta_{t,1}$ is decreasing on $(0,\infty)$ and then symmetry and continuity yield the claimed growth properties. Consider $x\in (0,\infty)$ and $v_t^{\varepsilon}(x):=\E\left[ f_{\varepsilon}(Y_x^-(t))\right]$ where \mbox{$f_{\varepsilon}(y)=1_{\{|y|<\varepsilon\}}$}. Here, $\beta_{t,1}(x)$ is defined as the density of $Y_x^{-}(t)$ at 0. Hence, $\beta_{t,1}(x) = p_{t}(x,0)= \lim_{\varepsilon \searrow 0} \frac{1}{\varepsilon}v_{t}^{\varepsilon}(x)$. Thus it is enough to show that $v_t^{\varepsilon} (x)$ is decreasing on $(0,\infty)$ for every ${\varepsilon}>0$. Let $0<x<y<\infty$. \mbox{Proposition \ref{p:stochineq}} yields $P(\forall t\geq 0: Y_y^{-}(t)\geq Y_x^{-}(t)) = 1.$
Define $\tau := \inf\{t>0: \ -Y_x^{-}(t)=Y_y^{-}(t) \}$. \cite[Proposition 2.1.5 a)]{ethier.kurtz.86} yields that $\tau$ is a stopping time because it is the first contact time with the closed set $\{0\}$ of the continuous process $Y_x^{-}+Y_y^{-}$. Observe, that $\vert Y_x^-(t)\vert \leq Y_y^-(t)$ for any $t\in[0,\tau]$. We can write
\begin{align*}
v_t^{\varepsilon}(y)-v_t^{\varepsilon}(x)&=\E\left[\left(1_{\{|Y_y^{-}(t)|<\varepsilon\}}-1_{\{|Y_x^{-}(t)|<\varepsilon\}}\right)1_{\{t<\tau\}}\right]\\
&+\E\left[\left(1_{\{|Y_y^{-}(t)|<\varepsilon\}}-1_{\{|Y_x^{-}(t)|<\varepsilon\}}\right)1_{\{t\geq\tau\}}\right]\\
&= C_1+C_2
\end{align*}
where $C_1:=\E\left[\left(1_{\{|Y_y^{-}(t)|<\varepsilon\}}-1_{\{|Y_x^{-}(t)|<\varepsilon\}}\right)1_{\{t<\tau\}}\right]$ and $C_2$ is the other summand.
It can be seen that $C_1$ is negative since \mbox{$P(|Y_x^{-}(t)|\leq\varepsilon, t<\tau)\geq P(|Y_y^{-}(t)|\leq \varepsilon, t<\tau)$}. For the term $C_2$ we use exactly the same Markov-argument as for the term $A_2$ in \mbox{Lemma \ref{l:density minus}} by defining $\tilde{\tau} := \tau \vee t$. Then $\tilde \tau\geq \tau$ and $\tilde \tau$ is $\mathcal{F}_{\tau}$-measurable. Thus, the strong Markov property of $Y_x^{-}$ and $Y_y^{-}$ and \cite[Corollary 2.6.18]{karatzas.shreve.91} yield
\begin{align*}
\E\left[1_{\{|Y_y^{-}(t)|<\varepsilon\}}1_{\{t\geq\tau\}}|\mathcal{F}_{\tau}\right]&= \E\left[1_{\{|Y_y^{-}(\tilde \tau)|<\varepsilon\}}1_{\{t\geq\tau\}}|\mathcal{F}_{\tau}\right]\\
&=1_{\{t\geq\tau\}}\E\left[1_{\{|Y_y^{-}(\tilde \tau)|<\varepsilon\}}|\mathcal{F}_{\tau}\right]\\
&=1_{\{t\geq\tau\}}\E\left[1_{\{|Y_y^{-}(\xi)|<\varepsilon\}}|_{\xi=\tilde\tau - \tau}\right]
\end{align*}
P-a.s. On the other hand, observe that $Y_y^{-}(\tau)=-Y_x^{-}(\tau)$ by the definition of $\tau$. So
$$1_{\{t\geq\tau\}}\E\left[1_{\{|Y_y^{-}(\xi)|<\varepsilon\}}|_{\xi=\tilde\tau - \tau}\right] = 1_{\{t\geq\tau\}}\E\left[1_{\{|-Y_x^{-}(\xi)|<\varepsilon\}}|_{\xi=\tilde\tau - \tau}\right]$$
which implies that $C_2=0$. As a result
$$v_t^{\varepsilon}(y)-v_t^{\varepsilon}(x) = \E\left[\left(1_{\{|Y_y^{-}(t)|<\varepsilon\}}-1_{\{|Y_x^{-}(t)|<\varepsilon\}}\right)1_{\{t<\tau\}}\right] \leq 0$$
which implies
$$\beta_t(y)-\beta_t(x) = \lim_{\varepsilon \searrow 0}\frac{1}{\varepsilon} (v_t^{\varepsilon}(y)-v_t^{\varepsilon}(x))\leq 0$$
for every $x,y\in \mathbb{R}$ with $0<x<y$.

Finally, we show that $\beta_{t,1}$ has exponential tails. Observe that \mbox{$|q_{t-s}(0,0)|\leq \frac{1}{\sqrt{2\pi (t-t/2)}}+1$} for $s\in [0,t/2]$ and thus
\begin{align*}
\int_0^{t/2} |q_{t-s}(0,0)|\frac{|x|}{\sqrt{2\pi s^3}}e^{-\frac{(|x|-s)^2}{2s}}\leq K|x| e^{|x|}\int_0^{t/2} s^{-3/2}e^{-\frac{|x|^2}{2s}}ds
\end{align*}
where $K$ denotes the collection of constants not depending on $x>0$. Moreover, one can show that
$$\int_0^{t/2} s^{-3/2}e^{-\frac{|x|^2}{2s}}ds\leq K \frac{1}{|x|^2}e^{-\frac{|x|^2}{2t}}$$
for a constant $K>0$ independent of $x$. Altogether
$$\int_0^{t/2} |q_{t-s}(0,0)|\frac{|x|}{\sqrt{2\pi s^3}}e^{-\frac{(|x|-s)^2}{2s}}\leq K \frac{e^{|x|}}{|x|}e^{-\frac{|x|^2}{2t}}.$$
Finally, $|\rho_{\tau_0^{x}}(s)| \leq K\vert x\vert e^{-\frac{(|x|-t)^2}{2t}}$ for $s\in [t/2,t]$, $\vert x\vert>t$ which yields
$$\int_{t/2}^{t}|q_{t-s}(0,0)| |\rho_{\tau_0^{x}}(s)| ds \leq K |x|e^{-\frac{(|x|-t)^2}{2t}}.$$
\end{proof}

From now on, let us consider the processes $Y_x^-$ and $Y_x^+$ given in Equation \eqref{e:signum SDE} for the multidimensional case, i.e. $x\in \mathbb{R}^d$, $\sgn(x):=\frac{x}{|x|}1_{\{x\neq 0\}}$ and $W$ a $d$-dimensional standard Brownian motion. We denote $x=(x_1,\dots, x_d)\in \mathbb{R}^d$, $W=(W_1,\dots, W_d)$ and $Y_x^{\pm}(t) = (Y_{x,1}^{\pm}(t), \dots, Y_{x,d}^{\pm}(t))$. Theorem \ref{t:main theorem, multi dimensional} guarantees that the density of any adapted process $X_u(t) := \int_0^t u(s)ds + W(t)$ , $u\in \mathcal{A}$, has bounds $\alpha_{d,t}:=\alpha_{d,t,1}$ and $\beta_{d,t}:=\beta_{d,t,1}$.

We start with a proposition which gives a different view on the functions $\alpha_{d,t,C}$ and $\beta_{d,t,C}$. Namely, we define $Z_x^{\pm}(t):=\vert Y_x^{\pm}(t)\vert^2$ with $Z_x^{\pm}(0) = \vert x\vert^2$ and denote $V_{\varepsilon}$ the volume of the $d$-dimensional Euclidean ball of radius $\varepsilon$ then we have
 $$ \alpha_{t,C}(x) = \limsup_{\epsilon\rightarrow 0}\frac{P(\vert Y_x^{+}(t)\vert\leq \epsilon)}{V_\epsilon}=\limsup_{\epsilon\rightarrow 0}\frac{P(Z_x^+(t)\leq \epsilon^2)}{C_d \ \epsilon^d},$$
 and
 $$ \beta_{t,C}(x) = \limsup_{\epsilon\rightarrow 0}\frac{P(\vert Y_x^{-}(t)\vert\leq \epsilon)}{V_\epsilon}=\limsup_{\epsilon\rightarrow 0}\frac{P(Z_x^-(t)\leq \epsilon^2)}{C_d \ \epsilon^d},$$
cf.\ Theorem \ref{t:main theorem, multi dimensional}, where $C_d := \frac{\pi^{d/2}}{\Gamma\left( \frac{d}{2}+1 \right)}$. In view of this equality, we are interested in the behaviour of the transition density of $(Z_x)_{x\in \mathbb{R}}$ near zero which will be exploited in Theorem \ref{t:multibounds} below.
\begin{prop}\label{p:alternative SDE}
Let $\fraum$ be a filtered probability space. Let $W$ be a $d$-dimensional Brownian motion and $Y_x^{\pm}$ be the solution to the SDE
  $$Y_x^{\pm}(t) := x \pm \int_0^t \sgn(Y_x^{\pm}(s)) ds + W(t),\quad t\geq0$$
 for any $x\in\mathbb R^d$. Define $Z_x^{\pm}(t) := \vert Y_x^{\pm}(t)\vert^2$, $B_x^{\pm}(t):=\int_0^t \sgn(Y_x^{\pm}(s))dW(s)$ for any $x\in\mathbb R^d$, $t\geq 0$. Then $(Z_x^{\pm},B_x^{\pm})$ is a solution to the SDE
  \begin{align}  \label{e:square SDE}
  dZ_x^{\pm}(t) = \left( d\pm 2\sqrt{Z_x^{\pm}(t)} \right)dt + 2\sqrt{Z_x^{\pm}(t)}dB_x^{\pm}(t),\quad Z_x^{\pm}(0)=\vert x\vert^2,\quad t\geq 0
  \end{align}
  for which pathwise uniqueness holds.
\end{prop}
\begin{proof}
  Let $f:\mathbb R^d\rightarrow\mathbb R_+,x\mapsto \vert x\vert^2$. Then $ Df(x)\cdot y= 2\<x,y\>$ and $H_2f(x) = 2 I_d$ for any $x,y\in\mathbb R^d$ where $Df$ denotes the Fr\'echet differential of $f$, $H_2f$ the Hessian matrix of $f$ and $I_d$ denotes the unit matrix in $\mathbb R^{d\times d}$. It\^o's formula yields
 \begin{align*}
   Z_x^{\pm}(t) &= \vert x\vert^2 + \int_0^t (\pm 2\<Y_x^{\pm}(s),\sgn(Y_x^{\pm}(s))\> + d) ds + \int_0^t 2 Y_x^{\pm}(s) dW(s)  \\
          &= \vert x\vert^2 + \int_0^t ( \pm 2\sqrt{Z_x^{\pm}(s)} + d) ds + \int_0^t 2 \sqrt{Z_x^{\pm}(s)} dB_x^{\pm}(s)
 \end{align*}
 for any $t\geq0$. Since $B_x^{\pm}$ is a Brownian motion, $(Z_x^{\pm},B_x^{\pm})$ is a weak solution as required.

 It remains to show that the SDE has unique weak solutions. Let $Q$ be a measure, equivalent to $P$, such that $\widetilde W^{\pm}(t) := B_x^{\pm}(t) - t$ is a standard $Q$-Brownian motion. Then the SDE can be rewritten as
 \begin{align}\label{e:modified SDE}
    dZ_x^{\pm}(t) = \left( d \right)dt \pm 2\sqrt{Z_x^{\pm}(t)}d\widetilde W^{\pm}(t),\quad Z_x^{\pm}(0)=\vert x\vert^2,\quad t\geq 0.
 \end{align}
 \cite[Theorem IX.3.5 ii)]{revuz.yor.99} yields that pathwise uniqueness holds for SDE \eqref{e:modified SDE} under $Q$.
\end{proof}

The following result gives explicit bounds for the functions $\alpha_{d,t}$ and $\beta_{d,t}$.\begin{thm}\label{t:multibounds}
We have
$$\frac{2^d}{C_d d^{d/2}}\prod_{i=1}^d\alpha_{1,t}(x_i) \leq \alpha_{d,t}(x)\leq \beta_{d,t}(x)\leq \frac{2^d}{C_d}\prod_{i=1}^d\beta_{1,t}(x_i), \quad x\in \mathbb{R}^d$$
where $C_d:= \frac{\pi^{d/2}}{\Gamma\left( \frac{d}{2}+1\right)}$.
\end{thm}
\begin{proof}
Since the proof is fairly similar for $\alpha_{d,t}$, we will just show the  last inequality.

Define the processes $Z_{x,i}^-(t) := |Y_{x,i}^-(t)|^2$, $i=1,\dots,d$. It\^{o}'s formula yields
\begin{align*}
Z_{x,i}^-(t) &= |x_i|^2 + \int_0^t \left( 1- 2\sqrt{Z_{x,i}^-(s)}\frac{|Y_{x,i}^-(s)|}{|Y_{x}^-(s)|}\right)ds + 2\int_0^t Y_{x,i}^-(s) dW_i(s)\\
&=|x_i|^2 + \int_0^t \left( 1- 2\sqrt{Z_{x,i}^-(s)}\frac{|Y_{x,i}^-(s)|}{|Y_{x}^-(s)|}\right)ds + 2\int_0^t \sqrt{Z_{x,i}^-(s)} dB_i(s)\\
&\geq |x_i|^2 + \int_0^t \left( 1- 2\sqrt{Z_{x,i}^-(s)}\right)ds + 2\int_0^t \sqrt{Z_{x,i}^-(s)} dB_i(s)
\end{align*} 
where $B_i(t) := \int_0^t \sgn(Y_{x,i}^-(s))dW_i(s)$ defines a new standard Brownian motion w.r.t. $P$. \cite[Theorem IV.3.6]{revuz.yor.99} and It\^{o} isometry ensure that $(B_1,\dots, B_d)$ is a $d$-dimensional standard Brownian motion. Let $V_i$ be the solution of the SDE
\begin{align}\label{eqvi}
V_i(t)=|x_i|^2 + \int_0^t \left( 1-2\sqrt{V_i(s)}\right)ds + 2\int_0^t \sqrt{V_i(s)}dB_i(s)
\end{align}
for any $i=1,\dots,d$ and $Q$ be the measure, equivalent to $P$, such that $\widetilde{B}(t) := B(t) - (t,\dots,t)$ is a $Q$-Brownian motion  where $B=(B_1,\dots,B_d)$. Then, we have
\begin{align*}
Z_{x,i}^-(t) &= |x_i|^2 + \int_0^t \left( 1+ 2\sqrt{Z_{x,i}^-(s)}\left(1- \frac{|Y_{x,i}^-(s)|}{|Y_{x}^-(s)|}\right)\right)ds + 2\int_0^t \sqrt{Z_{x,i}^-(s)} d\widetilde B_i(s),\\
V_i(t) &= |x_i|^2 + \int_0^t 1 ds + 2\int_0^t \sqrt{V_i(s)}d\widetilde B_i(s).
\end{align*}
Similar arguments as in the proof of \cite[Theorem IX.3.7]{revuz.yor.99} show that $Z_{x,i}^-(t) \geq V_i(t)$ for any $t\geq 0$, $Q$-a.s.

Observe that pathwise uniqueness holds for Equation \eqref{e:square SDE} by Proposition \ref{p:alternative SDE} and hence \cite[Theorem IX.1.7 ii)]{revuz.yor.99} states that $V_i$ is a strong solution to Equation \eqref{eqvi}. Consequently, $V_i$ is $\sigma(B_i)$-measurable and hence $V_1, \dots,V_d$ are independent processes.

Now given $a=(a_1,\dots, a_d)\in \mathbb{R}^d$ one has $|a|\geq \max{\{|a_i|, i=1,\dots,d\}}$. This implies
\begin{align*}
P(|Y_x^- (t)| \leq \varepsilon) &\leq P\left(\bigcap_{i=1}^d \{|Y_{x,i}^-(t)|\leq \varepsilon\}\right)\\
&= P\left(\bigcap_{i=1}^d \{Z_{x,i}^-(t)\leq \varepsilon^2\}\right)\\
&\leq \prod_{i=1}^d P\left( V_i(t)\leq \varepsilon^2\right)
\end{align*}
where in the last step we use the inequalities $Z_{x,i}^-(t) \geq V_i(t)$ for every $t\geq 0$, $P$-a.s. and the fact that $V_1,\dots, V_d$ are independent processes.

By Proposition \ref{p:alternative SDE} the law of $V_i(t)$ under $P$ is the same as the law of $|A_i(t)|^2$ under $P$ where
$$A_i(t)=x_i - \int_0^t \sgn(A_i(s))ds+W_i(t), \quad t\geq 0$$
and the law of $A_i(t)$ is given in Lemma \ref{l:density minus}. Hence, we have
\begin{align*}
\beta_{d,t}(x) &\xleftarrow{\varepsilon \to 0} \frac{P(|Y_x^-(t)|\leq \varepsilon)}{C_d\varepsilon^d}\\
&\leq  \frac{1}{C_d}\prod_{i=1}^d \frac{P(|A_i(t)|\leq \varepsilon)}{\varepsilon}\\& \xrightarrow{\varepsilon \to 0} \frac{2^d}{C_d}\prod_{i=1}^d \beta_{1,t}(x_i)
\end{align*}
for any $t>0$.
\end{proof}

In order to prepare our main result of this section we will start with a series of lemmas which aims at showing the continuity condition of \cite[Theorem IX.2.11]{js.87}. The needed continuity condition is summarised in Lemma \ref{l:continuity of B}.

\begin{lem}\label{l:Pfadabschaetzung}
 Let $\fraum$ be a filtered probability space. Let $\phi:\mathbb R_+\rightarrow [0,1]$ such that $\phi$ is infinitely differentiable, $\phi$ is constant $1$ on $[0,1]$ and constant $0$ on $[2,\infty)$. Define
  $$ A_k:\mathbb R_+\times \skor\rightarrow\mathbb R, (t,f)\mapsto \int_0^t \phi(k\vert f(s)\vert)ds $$
  for any $k\in\mathbb N$. Let $b:\Omega\times\mathbb R_+\rightarrow\mathbb R^d$ be an adapted process which is bounded by $1$, $x\in\mathbb R^d$ and define
   $$ X(t) := x + \int_0^t b(s)ds + W(t),\quad t\geq 0.$$
  Then $\E(A_k(t,X)) \leq \sqrt{tc_k(t)}\exp(t/2)$ where $c_k(t):=\E(A_k(t,x+W))\rightarrow 0$ for $k\rightarrow\infty$.
\end{lem}
\begin{proof}
 Define $Z(t):=\mathcal E(-\int_0^tb(s)dW(s))= 1-\int_0^t Z(s)b(s)dW(s)$ and $dQ\vert_{\mathcal F_t}:=Z_tdP\vert_{\mathcal F_t}$. Then Girsanov's theorem \cite[Theorem III.3.24]{js.87} yields that $X$ is a $Q$-Brownian motion starting in $x$. Define $Y(t):=1/Z(t)$. Then
  $$ Y(t) = 1 + \int_0^tY(s)b(s)dX(s),\quad t\geq0$$
 and hence by Gronwall's lemma, see e.g. \cite[Appendix $\S$1]{revuz.yor.99}, $\E_Q(Y(t)^2) \leq \exp(t)$. We have
 \begin{align*}
    \E(A_k(t,X)) &= \E_Q\left(A_k(t,X)Y(t)\right)  \\
                   &\leq \sqrt{\E_Q(A_k(t,X)^2)}\sqrt{\E_Q(Y^2(t))} \\
                   &\leq \sqrt{t\E_Q(A_k(t,X))}\exp(t/2) \\
                   &= \sqrt{tc_k(t)}\exp(t/2)
 \end{align*}
 for any $t\geq 0$ where we used the Cauchy-Schwartz inequality twice and the fact that $\varphi^2\leq \varphi$. We have
  \begin{align*}
     c_k(t) &= \E(A_k(t,x+W)) \\
            &\rightarrow \E(\lambda(\{s\in[0,t]:x+W(s)=0\}) \\
            &= 0
  \end{align*}
  for $k\rightarrow\infty$ by Lebesgue's dominated convergence theorem where $\lambda$ denotes the Lebesgue measure on $\mathbb{R}$.
\end{proof}

\begin{lem}\label{l:An abschaetzung}
 Let $\fraum$ be a filtered probability space and $W$ be a $d$-dimensional standard Brownian motion. Let $(A_k)_{k\in\mathbb N}$ and $(c_k)_{k\in\mathbb N}$ be as in Lemma \ref{l:Pfadabschaetzung}. Let $(M_n)_{n\in\mathbb N}$ be a sequence of processes that converges in probability to $W$. For any $n\in\mathbb N$ let $b_n:\Omega\times\mathbb R_+\rightarrow\mathbb R^d$ be an adapted process which is bounded by $1$, $x\in\mathbb R^d$ and define
 \begin{align*}
   X_n(t) := x + \int_0^t b_n(s)ds + M_n(t),\quad t\geq0.
 \end{align*}
 Also, assume that $(X_n)_{n\in\mathbb N}$ converges in distribution to some process $X_\infty$. 
 
 Then $X_\infty$ has $P$-a.s.\ continuous sample paths and
  $$ \E A_k(t,X_\infty) \leq \sqrt{tc_k(t)}\exp(t/2),\quad t\geq 0, \quad k\in\mathbb N.$$
 Moreover, $\lambda\{s\in\mathbb R_+:X_\infty(s)=0\}=0$ $P$-a.s.\ where $\lambda$ denotes the Lebesgue measure on $\mathbb{R}$.
\end{lem}
\begin{proof}
 Define $Y_n(t) := x + \int_0^t b_n(s)ds + W(t)$, $t\geq0$. Then
 $$ X_n - Y_n = M_n - W \rightarrow 0 $$
 in probability for $n\rightarrow\infty$. Hence, $Y_n \rightarrow X_\infty$ in distribution. Since $Y_n$ has continuous sample paths for any $n\in\mathbb N$, $X_\infty$ has $P$-a.s.\ continuous sample paths. Let $t,\epsilon >0$. Since $A_k$ is continuous we have $\E A_k(t,X_\infty)\leq \epsilon + \E A_k(t,Y_n)$ for some $n\in\mathbb N$. Hence, by Lemma \ref{l:Pfadabschaetzung} we have
  \begin{align*}
     \E A_k(t,X_\infty) &\leq \epsilon + \E A_k(t,Y_n) \\
                        &\leq \epsilon + \sqrt{tc_k(t)}\exp(t/2).
  \end{align*}
 Thus, we have
   \begin{align*}
     \E(\lambda\{s\in[0,t]:X_\infty(s)=0\}) &\leftarrow \E A_k(t,X_\infty) \\
                                             &\rightarrow 0
   \end{align*}
  for $k\rightarrow\infty$ and $t\geq 0$. Thus $\E(\lambda\{s\in\mathbb R_+:X_\infty(s)=0\})\leq \sum_{n=1}^\infty\E(\lambda\{s\in[0,n]:X_\infty(s)=0\})=0$. The claim follows.
 \end{proof}

\begin{rem}\label{r:inequality}
 Let $x\in\mathbb R^d\backslash\{0\}$, $\epsilon\in (0,\vert x\vert)$ and $y\in\mathbb R^d$ such that $\vert x-y\vert\leq\epsilon$. Then
 $$ \vert \sgn(x)-\sgn(y)\vert \leq \sqrt{2}\left(\frac{\epsilon}{\vert x\vert}\right). $$
\end{rem}

\begin{lem}\label{l:continuity of B}
 Let $\fraum$ be a filtered proability space and $W$ be a $d$-dimensional standard Brownian motion. Let $(b_n)_{n\in\mathbb N}$ be adapted processes which are bounded by $1$. Let $x\in\mathbb R^d$, $(M_n)_{n\in\mathbb N}$ be a sequence of adapted processes which converges in probability to $W$ and define
 $$ X_n(t) := x + \int_0^tb_n(s)ds + M_n(t).$$
 Assume that $(X_n)_{n\in\mathbb N}$ converges in distribution to some process $X_\infty$ and define
 $$ B:\mathbb R_+\times \skor\rightarrow\mathbb R^d,(t,f)\mapsto -\int_0^t\sgn(f(s))ds,\quad t\geq 0 $$
  
 Then $f\mapsto B(t,f)$ is $P^{X_\infty}$-a.s.\ continuous for any $t\geq 0$.
\end{lem}
\begin{proof}
 Let $A_k$ be as in Lemma \ref{l:Pfadabschaetzung} for any $k\in\mathbb N$. Lemma \ref{l:An abschaetzung} yields
  $$\lambda\{s\in\mathbb R_+:X_\infty(s)=0\}=0$$ 
  $P$-a.s. and hence $A_k(X_\infty,t)\rightarrow 0$ for $k\rightarrow\infty$ $P$-a.s.

Let $t\geq 0$ and $f,g_k\in\skor$ such that $\sup\{\vert f(s)-g_k(s)\vert:s\leq t\}\leq 1/k^2$ for any $k\in\mathbb N$. Then, we have
 \begin{align*}
   \vert B(t,f) - B(t,g_k) \vert \leq& \int_0^t \vert\sgn(f(s)) - \sgn(g_k(s))\vert ds \\
   						=& \int_0^t \vert\sgn(f(s)) - \sgn(g_k(s))\vert 1_{\{\vert f(s)\vert\leq 1/k\}} ds\\
   						&+ \int_0^t \vert\sgn(f(s)) - \sgn(g_k(s))\vert 1_{\{\vert f(s)\vert > 1/k\}} ds \\
                         \leq& 2\int_0^t 1_{\{\vert f(s)\vert\leq 1/k\}} ds + t\sqrt{2}/k \\
                         \leq& 2\int_0^t \phi(k\vert f(s)\vert) ds + t\sqrt{2}/k\\
                         =& 2A_k(t,f) + t\sqrt{2}/k \\
                         \rightarrow& 0
 \end{align*}
 $P^{X_\infty}$-a.s.\ for $k\rightarrow\infty$ where we used the integral inequality, then we split the support of $f$,  Remark \ref{r:inequality} with $\epsilon=1/k^2$ and the inequality $1_{[0,1]}(x)\leq \phi(x)$ for any $x\geq 0$.
\end{proof}

In the next lemma the martingales $M_n$ converge to the Brownian motion $W$ but they, and hence the drift in $X_n$, are not adapted to the same Brownian motion. We show that they converge in our specific set-up.
\begin{lem}\label{l:convergence of the marginals}
 Let $\fraum$ be a filtered probability space. Let $W$ be a $d$-dimensional Brownian motion, $x\in\mathbb R^d$ and $M_n(t) := W(\theta_n(t))$ where $\theta_n(t):=\inf\{k/n:t< k/n\}$ for any $n\in\mathbb N$. Assume that $X_n(t) = x - \int_0^t \sgn(X_n(s))ds + M_n(t)$ for any $n\in\mathbb N$. Then $(X_n)_{n\in\mathbb N}$ converges in distribution to the solution $X$ of the SDE
 \begin{align}\label{e:minus signum SDE}
  X(t) = x - \int_0^t\sgn(X(s))ds + W(t),\quad t\geq 0. 
 \end{align}
\end{lem}
\begin{proof}
 By an independent enlargement of $\mathcal F_0$-argument, we may assume that there is a sequence $(H_n)_{n\in\mathbb N}$ of random variables which are indepent of $W$, $\mathcal F_0$-measurable and that $H_n$ is centered normal on $\mathbb R^d$ with variance $I_d/n$ where $I_d$ denotes the identity matrix in $\mathbb R^{d\times d}$.

 Define $\tilde\theta_n(t):=\theta_n(t)-1/n=\max\{k/n:k\geq 0, k/n\leq t\}$ for any $n\in\mathbb N$. Then $0\leq \tilde\theta_n(t) \leq t$. Define the $\mathcal F$-adapted process $\widetilde M_n(t) := H_n + W(\widetilde\theta_n(t))$ and $\widetilde X_n(t) = x - \int_0^t \sgn(\widetilde X_n(s))ds + \widetilde M_n(t)$. Then $M_n$ has the same law as $\widetilde M_n$ and, consequently, $(X_n,M_n)$ has the same law as $(\widetilde X_n,\widetilde M_n)$ for any $n\in\mathbb N$. Moreover, $\widetilde M_n\rightarrow W$ in probability. 

 Define 
 \begin{align*}
  B(t,f) &:= -\int_0^t \sgn(f(s)) ds,\\
  C(t,f) &:= tI_d, \\
  \nu(A\times I)  &:=0,\\
  B_n(t) &:= B(t,\widetilde X_n),\\
  C_n(t) &:= 0I_d=0,\\
  \nu_n(A\times I) &:= \mu_n(A)\sum_{k=1}^\infty \delta_{k/n}(I)\\
 \end{align*}
 for any $t\in\mathbb R_+$, $f\in\skor$, $n\in\mathbb N$, $A\in\mathcal B(\mathbb R^d)$, $I\in\mathcal B(\mathbb R_+)$ where $\mu_n$ is the centered normal law with covariance matrix $I_d/n$. Then $(B_n,C_n,\nu_n)$ is the semimartingale characteristics of $X_n$ in the sense of \cite[Definition II.2.6]{js.87} relative to the truncation function $h(x):=\sgn(x)(\vert x\vert\wedge 1)$, $x\in\mathbb R^d$. Observe that $(B_n,C_n,\nu_n)_{n\in\mathbb N}$ and $(B,C,\nu)$ fulfil the conditions $[\mathrm{Sup-}\beta_7]$, $[\mathrm{Sup-}\gamma_7]$ and $[\mathrm{Sup-}\delta_{7,1}]$ in the sense of \cite[page 535]{js.87}. Thus \cite[Theorem IX.3.9]{js.87} states that $(\widetilde X_n)_{n\in\mathbb N}$ is tight. Let $(\widetilde X_{n_j})_{j\in\mathbb N}$ be a subsequence of $(\widetilde X_n)_{n\in\mathbb N}$ which converges in law and denote the limiting law by $P_\infty$. Lemma \ref{l:continuity of B} yields that $B$ is $P_\infty$-a.s.\ continuous. Let $Y$ be the canonical process on the canonical space $(\skor,(\mathcal G_t)_{t\geq 0},\mathcal B(\skor))$. Then \cite[Theorem IX.2.11]{js.87} yields that $Y$ is, under $P_\infty$, a semimartingale with characteristics $(B,C,\nu)$. The continuous martingale part, denote it $\widetilde W$, of $Y$ is a standard Brownian motion because its semimartingale characteristics is $(0,C,0)$. Moreover,
 $$ Y(t) = x + B(t,Y) + \widetilde W(t) = x - \int_0^t \sgn(Y(s))ds + \widetilde W(t),\quad t\geq 0. $$

 Thus $(Y,\widetilde W)$ is a weak solution to the SDE \eqref{e:minus signum SDE}. \cite[Corollary IX.1.12)]{revuz.yor.99} yields that the law $P_\infty$ of $Y$ coincides with the law of the solution $X$ of the SDE \eqref{e:minus signum SDE}. Consequently, any convergent subsequence of  $(\widetilde X_{n})_{n\in\mathbb N}$ converges in law to $X$. Since $(\widetilde X_n)_{n\in\mathbb N}$ is additionally tight, it, and hence $(X_n)_{n\in\mathbb N}$, converges to $X$.
\end{proof}

We are now in readiness to prove the core result of this section which is to solve a control problem. For dimension one, this problem has been studied by V. E. Bene\v{s} in \cite{benes.74} in the Markovian setting whose optimal control is indeed the \emph{signum} function in dimension one. Such solutions are known as \emph{bang-bang} solutions. Nevertheless, here we stress the fact that our case deals with multivalued controls, thus not \emph{bang-bang}, in addition to the non-Markovian setting as in the above mentioned result. For this reason we include a short proof based on the previous limit results.

\begin{thm}\label{t:control problem}
 Let $\mathcal{A}_{+}$ and $\mathcal{A}$ be as in the beginning of Section \ref{main results}. Let $T,\epsilon>0$, $x\in\mathbb R^d$ and define $u_x^* (t):=\sgn(Y_x^{+}(t))$ and $v_x^*(t):=-\sgn(Y_x^{-}(t))$. Then 
\begin{align}\label{OPplus}
     \inf_{u\in\mathcal A}P(\vert X_u(T)\vert \leq \epsilon) = P(\vert X_{u_x^*}(T)\vert \leq \epsilon)
  \end{align}  
 where $X_u(t) := x + \int_0^t u(s)ds + W(t)$ for $u\in \mathcal A$. In other words, an optimal control for the control problem above is given by $u_x^*$.
Similarly,  
  \begin{align}\label{OPminus}
     \sup_{v\in\mathcal A}P(\vert X_v(T)\vert \leq \epsilon) = P(\vert X_{v_x^*}(T)\vert \leq \epsilon).
  \end{align}
\end{thm}
\begin{rem}
The control problem given in (\ref{OPplus}) can be interpreted as follows: one wishes to find the stochastic process among those in $\mathcal{A}$ that minimises the probability that the underlying process $X$ is near zero. In other words, we want the process $X_{u}(T)$ to escape from 0 as much as possible. Intuitively, the process $Y_x^{+}$ is doing that. Whenever $Y_x^{+}(t)$ is near zero on the positive line, the drift $\sgn (Y_x^{+}(t))$ is positive and pushes $Y_x^{+}(t)$ even further away up and if $Y_x^{+}(t)$ is near zero from below the drift is negative and sends $Y_x^{+}(t)$ further down. For the control problem in (\ref{OPminus}) the idea is similar, but there one wishes to maximise the probability of being close to zero, which $-\sgn( Y_x^{-}(t))$ clearly does.

For a general reference on control problems we relate to {\O}ksendal and Sulem \cite{oeksendal.sulem.07}.
\end{rem}
\begin{proof}[Proof of Theorem \ref{t:control problem}]
For the sake of brevity we will only show the proof of the control for (\ref{OPminus}).

For any $n\in \mathbb N$ define $\theta_n(t):=\inf\{Tk/n:k\in\mathbb N,t<T k/n\}$, $M_n(t):=W(\theta_n(t))$ and
 $$ \mathcal A_n:=\{v\in\mathcal A_+: v(t)\text{ is }\F_{\theta_n(t)}\text{-measurable for any }t\in[0,T]\}$$
 Then $M_n$ is adapted to the filtration $(\mathcal G_{n,t})_{t\geq 0}:=(\F_{\theta_n(t)})_{t\geq 0}$.

 Let $X_n(t) = x - \int_0^t\sgn(X_n(s))ds + M_n(t)$, $t\geq 0$. A simple backward induction yields that
  $$ P(\vert X_n(T-)\vert \leq \epsilon) = \sup_{v\in\mathcal A_n} P\left(\left\vert x+\int_0^Tv(s)ds + M_n(T-)\right\vert\leq\epsilon\right).$$

Lemma \ref{l:convergence of the marginals} yields that $(X_n)_{n\in\mathbb N}$ converges in law to $Y_x^-$. Since $Y_x^-(T)$ has no atoms, we have $P(\vert X_n(T)\vert\leq \epsilon)\rightarrow P(\vert Y_x^-(T)\vert\leq\epsilon)$ for $n\rightarrow\infty$. Thus, we have
 \begin{align*}
   P(\vert Y_x^-(T)\vert\leq \epsilon) &\leq \sup_{v\in\mathcal A}P(\vert X_v(T)\vert \leq \epsilon) \\
   &\leq \sup_{v\in \mathcal{A}_n}P(\vert X_v(T)\vert \leq \epsilon)\\
                                    &\leq \sup_{v\in\mathcal A_n}P\left(\left\vert x+\int_0^Tv(s)ds + M_n(T-)\right\vert \leq \epsilon\right) \\
                                    &= P(\vert X_n(T-)\vert\leq \epsilon) \\
                                    &\rightarrow P(\vert Y_x^-(T)\vert\leq\epsilon)
 \end{align*}
 for $n\rightarrow\infty$. Thus $v_x^*$ is an optimal control.

\end{proof}

Finally, we give the proof of our main result Theorem \ref{t:main theorem, multi dimensional}.
\begin{proof}[Proof of Theorem \ref{t:main theorem, multi dimensional}]
  Define $\widetilde X(t) := CX(t/C^2)$, $\widetilde u(t):=u(t/C^2)$ and the Brownian motion $\widetilde W(t):=CW(t/C^2)$. Then
  \begin{align*}
    \widetilde X(t) &= \int_0^{t/C^2} C^2u(s) ds + \widetilde W(t) \\   
                      &= \int_0^t\widetilde u(s) ds + \widetilde W(t)
  \end{align*}
 for any $t\geq 0$. Theorem \ref{t:control problem} states that
  $$ P(\vert \widetilde X(T)+x\vert \leq \epsilon) \leq P(\vert Y_x^-(T)\vert\leq\epsilon) $$
 for any $\epsilon,T>0$, $x\in\mathbb R^d$ and $u\in \mathcal{A}$. By definition
 $$ \lim_{\epsilon\rightarrow0}\frac{P(\vert Y_x^-(T)\vert\leq\epsilon)}{V_\epsilon} =\beta_{d,T,1}(x).$$
 Thus we have
 $$ \rho_{C,T}(x):=\limsup_{\epsilon\rightarrow0}\frac{P(\vert \widetilde X(T)-x\vert \leq \epsilon)}{V_\epsilon} \leq \beta_{d,T,1}(-x).$$
Observe that for any orthonormal transformation $U:\mathbb{R}^d\rightarrow \mathbb{R}^d$ we have
$$UY_x^{-}(t) = Ux - \int_0^t \sgn (UY_x^-(s)) ds + UW(t)$$
where here $UW$ is a standard Brownian motion and hence $(UY_x^-, UW)$ is a weak solution of \eqref{e:signum SDE} for $\pm = -$. Consequently, $UY_x^- (t)$ has the same law as $Y_{Ux}^-$ which implies $\beta_{d,T,1}(Ux) =\beta_{d,T,1}(x)$. Hence, we have
$$\rho_{C,T}(x) \leq \beta_{d,T,1}(x).$$
Lebesgue differentiation theorem \cite[Corollary 2.1.16]{grafakos.08} yields that $\rho_{C,T}$ is a version of the Lebesgue density of $\widetilde X(T)$. Consequently, the density $\rho_T$ of $X(T)$ given by
  $$ \rho_T(x) := \limsup_{\epsilon\rightarrow0}\frac{P(\vert X(T)-x\vert \leq \epsilon)}{V_\epsilon} $$
 satisfies
  $$ \rho_T(x) \leq \beta_{d,T,C}(x). $$ 
 Analogue arguments show that
 $$ \alpha_{d,T,C}(x) \leq \rho_T(x).$$
\end{proof}


\begin{thebibliography}{1}
\bibitem{applebaum.09}
D.~Applebaum, \emph{L\'evy processes and stochastic calculus}, second ed.,
  Cambridge University Press, Cambridge, 2009.

\bibitem{bally.caramellino.11}
V.~Bally and L.~Caramellino, \emph{Riesz transform and integration by parts
  formulas for random variables}, Stochastic Processes and their Applications
  \textbf{121} (2001), 1332--1355.

\bibitem{bally.kohatsu-higa.10}
V.~Bally and A.~Kohatsu-Higa, \emph{Lower bounds for densities of asian type
  stochastic differential equations}, Journal of Functional Analysis
  \textbf{258} (2010), 3134--3164.

\bibitem{banos.nilssen.14}
D.~Ba\~{n}os and T.~Nilssen, \emph{Malliavin regularity and regularity of
  densities of sdes. a classical solution to the stochastic transport
  equation}, To appear, 2014.

\bibitem{benes.74} V. E. Bene\v{s}, \emph{Girsanov functionals and optimal Bang-Bang laws for final value stochastic control}, Stochastic Processes and their Applications \textbf{2} (1974), 127--140.

\bibitem{borodin.salminen.96}
A.~N. Borodin and P.~Salminen, \emph{Handbook of brownian motion - facts and
  formulae}, Birkh\"{a}user Verlag, Basel, 1996.

\bibitem{bouleau.hirsch.86}
N.~Bouleau and F.~Hirsch, \emph{Propri\'{e}t\'{e}s d'absolue continuit\'{e}
  dans les espaces de Dirichlet et applications aux \'{e}quations
  diff\'{e}rentielles stochastiques}, Seminaire de probabilit\'{e}s XX, Lecture
  Notes in Mathematics, vol. 1204, Springer, Berlin, 1986, pp.~131--161.

\bibitem{debussche.fournier.13} A. Debussche, N. Fournier, \emph{Existence of densities for stable-like driven SDE's with H\"{o}lder continuous coefficients}, Journal of Functional Analysis \textbf{264} (2013), 1757--1778

\bibitem{de.marco.11} S. De Marco, \emph{Smoothness and asymptotic estimates of densities for SDEs with locally smooth coefficients and applications to square root-type diffusions}, The Annals of Applied Probability (2011), Vol. 21, No. 4, 1282--1321


\bibitem{ethier.kurtz.86}
S.~Ethier and T.~Kurtz, \emph{Markov {P}rocesses. {C}haracterization and
  {C}onvergence}, Wiley, New York, 1986.

\bibitem{grafakos.08}
L.~Grafakos, \emph{Classical Fourier analysis}, second ed., Springer, New York,
  2008.

\bibitem{hayashi.et.al.12}
M.~Hayashi, A.~Kohatsu-Higa, and G.~Y\^{u}ki, \emph{Local h\"{o}lder continuity
  property of the densities of solutions of sdes with singular coefficients},
  Journal of Theoretical Probability \textbf{26} (1991), no.~2013, 1117--1134.

\bibitem{hormander.69}
L.~H\"{o}rmander, \emph{Hypoelliptic second order differential equations}, Acta
  Mathematica \textbf{119} (1967), no.~1.

\bibitem{js.87}
J.~Jacod and A.~Shiryaev, \emph{Limit {T}heorems for {S}tochastic {P}rocesses},
  second ed., Springer, Berlin, 2003.

\bibitem{karatzas.shreve.91}
I.~Karatzas and S.~Shreve, \emph{Brownian motion and stochastic calculus},
  second ed., Springer, New York, 1991.

\bibitem{kohatsu.makhlouf.13}
A.~Kohatsu-Higa and A.~Makhlouf, \emph{Estimates for the density of functionals
  of sdes with irregular drift}, Stochastic Processes and their Applications
  \textbf{123} (2013), no.~5, 1716--1728.

\bibitem{kusuoka.stroock.82}
S.~Kusuoka and D.~Stroock, \emph{Application of the malliavin calculus, part
  I}, Proceedings of the Taniguchi Intern. Symp. on Stochastic Analysis (Kyoto
  and Katata), 1982, pp.~271--306.

\bibitem{malliavin.78}
P.~Malliavin, \emph{Stochastic calculus of variation and hypoelliptic
  operators}, Proceedings of the International Symposium on Stochastic
  Differential Equations (New York), Wiley, 1978, pp.~195--263.

\bibitem{nualart.zakai.89}
D.~Nualart and M.~Zakai, \emph{The partial {M}alliavin calculus}, Seminaire   de Probabilit\'{e}s XXIII (M.~Frittelli et~al., ed.), Lecture Notes in
  Mathematics, vol. 1372, Springer, Berlin, 2004, pp.~362--381.

\bibitem{oeksendal.03}
B.~{\O}ksendal, \emph{Stochastic differential equations. {A}n introduction with
  applications.}, sixth ed., Springer, Berlin, 2003.

\bibitem{oeksendal.sulem.07}
B.~{\O}ksendal and A.~Sulem, \emph{Applied stochastic control of jump
  diffusions}, second ed., Springer, Berlin, 2007.

\bibitem{revuz.yor.99}
D.~Revuz and M.~Yor, \emph{Continuous martingales and Brownian motion}, third
  ed., Springer, Berlin, 1999.
  
\bibitem{veretennikov.79} A.Y. Veretennikov, \emph{On the strong solutions of stochastic differential equations}. Theory of Probabability and its Applications \textbf{24} (1979), 354--366.  
  
\end{thebibliography}

\end{document}